\documentclass{article}

\usepackage{setspace}
\usepackage{amssymb,amsmath,amsthm,tabu,longtable}
\usepackage{subfigure}
\usepackage{calc}
\usepackage{verbatim}
\usepackage{enumerate}
\usepackage{cite}
\usepackage{epsfig}
\usepackage{graphicx}
\usepackage{graphics}
\usepackage {tikz}
\usetikzlibrary{automata,arrows,calc,positioning}
\usetikzlibrary{decorations.pathreplacing}
\usepackage{xcolor}\usepackage[hmargin=2cm,vmargin=2.5cm]{geometry}
\usepackage{chngcntr}
\counterwithin{figure}{section}
\newcommand{\tiltfrac}[2]{\left.#1\middle/#2\right.}
\newtheorem{definition}{Definition}[section]
\newtheorem{lemma}[definition]{Lemma}
\newtheorem{proposition}[definition]{Proposition}
\newtheorem{theorem}[definition]{Theorem}
\newtheorem{corollary}[definition]{Corollary}

\newtheorem{conjecture}[definition]{Conjecture}

                                {\null\hfill$\Box$\par\medskip\medskip\medskip\medskip}

\newcommand*{\QED}{\null\hfill$\Box$\par\medskip}%
\newcommand{\av}{\text{\rm{av}}}
\newcommand{\avd}{\text{\rm{avd}}}
\newcommand{\navd}{\widehat{\text{\rm{avd}}}}
\newcommand{\Priv}[2]{\text{Priv}_{#1}(#2)}

\newcounter{eqcount}

\makeatletter
\renewcommand{\pod}[1]{\mathchoice
  {\allowbreak \if@display \mkern 18mu\else \mkern 8mu\fi (#1)}
  {\allowbreak \if@display \mkern 18mu\else \mkern 8mu\fi (#1)}
  {\mkern4mu(#1)}
  {\mkern4mu(#1)}
}

\title{The Average Order of Dominating Sets of a Graph}

\author{Iain Beaton\\
\small Department of Mathematics \& Statistics\\[-0.8ex]
\small Dalhousie University\\[-0.8ex] 
\small Halifax, CA\\
\small\tt ibeaton@dal.ca\\
\and
Jason I. Brown\thanks{Supported by NSERC grant RGPIN 2018-05227}\\
\small Department of Mathematics \& Statistics\\[-0.8ex]
\small Dalhousie University\\[-0.8ex] 
\small Halifax, CA\\
\small\tt Jason.Brown@dal.ca\\
}

\begin{document}

\tikzset{bignode/.style={minimum size=3em,}}
\maketitle

\begin{abstract}

This papers focuses on the average order of dominating sets of a graph. We find the extremal graphs for the maximum and minimum value over all graphs on  $n$ vertices, while for trees we prove that the star minimizes the average order of dominating sets. We prove the average order of dominating sets in graphs without isolated vertices is at most $3n/4$, but provide evidence that the actual upper bound is $2n/3$. Finally, we show that the normalized average, while dense in $[1/2,1]$,  tends to $\frac{1}{2}$ for almost all graphs.

\end{abstract}

\setstretch{1.4}

\section{Introduction}\label{sec:intro}

For a graph $G$ containing vertex $v$, $N_G(v) = \{u|uv \in E(G)\}$ denotes the \emph{open neighbourhood} of $v$ while $N_G[v] = N(v) \bigcup \{v\}$ denotes the \emph{closed neighbourhood} of $v$ (we will omit the subscript $G$ when only referring to one graph). For $S \subseteq V(G)$, the closed neighbourhood $N[S]$ of $S$ is simply the union of the closed neighbourhoods for each vertex in $S$. A subset of vertices $S$ is a \emph{dominating set} of $G$ if $N[S] = V(G)$, that is, every vertex is either in $S$ or adjacent to a vertex in $S$. The \emph{domination number} of $G$, denoted $\gamma(G)$, is the order of the smallest dominating set of $G$. The study of dominating sets in graphs is extensive (see, for example, \cite{hedetniemi}).
%A dominating with order $\gamma(G)$ is called a \emph{minimum dominating} set. 

Let $\mathcal{D}(G)$ denote the collection of dominating sets of $G$. Furthermore let $d_k(G) = |\{S \in \mathcal{D}(G): |S|=k\}|$. Then the \emph{average order of dominating sets} in $G$, denoted $\avd(G)$, is
$$\avd(G)= \frac{\sum\limits_{k=\gamma(G)}^{|V(G)|} kd_k(G)}{\sum\limits_{k=\gamma(G)}^{|V(G)|} d_k(G)},$$
that is, the average cardinality of a dominating set of $G$.

For graphs with few dominating sets $\avd(G)$ is relatively easy to compute using the above formula. For example the empty graph $\overline{K_n}$ has exactly one dominating set of order $n$, hence $\avd(\overline{K_n})=n$. However if $G$ has many dominating sets other techniques may be more appropriate to compute $\avd(G)$. The \emph{domination polynomial} of $G$ is defined by
$$D(G,x) = \sum_{k=\gamma (G)}^{|V(G)|} d_k(G)x^k,$$
(see \cite{2012AlikhaniPHD}, for example, for for a discussion of domination polynomials). The average order of dominating sets in $G$ can be regarded as the logarithmic derivative of $D(G,x)$ evaluated at 1, that is,
\[ \avd(G)=\frac{d}{dx} \ln (D(G,x)) \bigg|_{x=1} =  \frac{D'(G,1)}{D(G,1)}. \refstepcounter{eqcount} \label{eqn:avdG} \tag{\theeqcount} \]
This allows us to compute $\avd(G)$ quickly when $D(G,x)$ readily available. For example,
$$D(K_n,x)=(1+x)^n-1 \hspace{2mm} \text{ and } \hspace{2mm}D(K_{1,n-1},x)=x(1+x)^{n-1}+x^{n-1},$$
so
$$\avd(K_n)=\frac{n2^{n-1}}{2^n-1} \hspace{2mm} \text{ and } \hspace{2mm} \avd(K_{1,n-1})=\frac{(n+1)2^{n-2}+n-1}{2^{n-1}+1}.$$

It is trivial to observe that the domination polynomial is multiplicative over components, that is, for graphs $G$ and $H$, $D(G \cup H,x) = D(G,x)D(H,x)$ where $G \cup H$ denotes the disjoint union of $G$ and $H$. From this we can obtain a fundamental result which states that the average order of dominating sets is additive over components.

\begin{lemma}
\label{lem:disjointunion}
Let $G$ and $H$ be graphs. Then $\avd(G \cup H) = \avd(G) + \avd(H)$.
\end{lemma}

\begin{proof}
As $D(G \cup H,x) = D(G,x)D(H,x)$, it follows that $D'(G \cup H,x) = D'(G,x)D(H,x)+D(G,x)D'(H,x)$. Therefore 

$$\avd(G \cup H) = \frac{D'(G,1)D(H,1)+D(G,1)D'(H,1)}{D(G,1)D(H,1)} = \frac{D'(G,1)}{D(G,1)}+\frac{D'(H,1)}{D(H,1)}  =\avd(G) + \avd(H).$$
\end{proof}

Although the average order of dominating sets is a novel area of research, there has been work done on averages of serveral other graphs invariants: 
\begin{itemize}
\item Closely related to the Weiner Index of a graph \cite{1947Wiener}, the mean distance (between vertices) in a graph was introduced in 1977 by Doyle and Graver \cite{1977Doyle}. Doyle and Graver showed for connected graphs of order $n$ (that is, with $n$ vertices) the mean distance in a graph was maximized by a path, with mean distance $\tiltfrac{(n+1)}{3}$, and minimized by the complete graph, with mean distance 1. 
\item The mean subtree order of a graph was introduced in 1983 by Jamison \cite{1983Jamison}. Jamison showed for any tree $T$ on $n$ vertices, the average number of vertices in a subtree of $T$ is at least $\tiltfrac{(n+2)}{3}$, with that minimum achieved if and only if $T$ is a path. As the mean subtree order of $T$ is at most $n$, Jamison naturally defined the mean subtree order of $T$ divided by $n$ to be the density of $T$ and showed there were trees whose density approached 1 as $n \rightarrow \infty$. Jamison conjectured the tree with maximum density was some caterpillar graph.
% and Mol and Oellermann \cite{2019Mol} have recently shown that the number of leaves in the tree of maximum density is $O(\log_2n)$. However the exact structure of the tree of maximum density remains an open problem. 
Additionally, the mean subtree order has been subject to a fair amount of recent work \cite{1984Jamison, 2015Wagner, 2010Vince, 2014Haslegrave}. 
\item The average size of an independent set in a graph was introduced in 2019 by Andriantiana et. el. \cite{2019AverageIndepedence}. Andriantiana et. el. showed the average number of vertices of an independent set in a graph was maximized by the empty graph and minimized by the complete graph. They also showed the average number of vertices of an independent set in a tree was maximized by $P_n$ and minimized by $K_{1,n-1}$. 
\item In 2020, Andriantiana et. el. \cite{2020Matchings} introduced the average size of a matching in a graph was introduced. Andriantiana et. el. showed the average number of edges in a matching of a graph was now minimized by the empty graph and maximized by the complete graph. They also showed the average number of edges in a matching in a tree was maximized by $P_n$ and minimized by $K_{1,n-1}$.
\end{itemize}

%The domination polynomial is conjectured to be unimodal by Alikhani \cite{2012AlikhaniPHD}. 
We remark that when the domination polynomial has all real roots, the average order of domination sets of graph $G$ can also determine the mode of the coefficients of $D(G,x)$; Darroch \cite{1964Darroch} showed in general that a positive sequence $(a_0,a_1,\cdots ,a_n)$ has its mode at either $\left\lfloor \frac{f'(1)}{f(1)} \right\rfloor$ or $\left\lceil \frac{f'(1)}{f(1)} \right\rceil$, where $f(x) = a_0 + a_1x + \cdots + a_nx^n$ is the associated generating polynomial. Therefore by $(\ref{eqn:avdG})$, if $D(G,x)$ has all real roots then it mode is at $\left\lfloor \avd(G) \right\rfloor$ or $\left\lceil \avd(G) \right\rceil$. 
%(Oboudi \cite{2015Oboudi} conjectured the exactly structure of $G$ when $D(G,x)$ has all real roots. We will discuss this family of graphs in Section \ref{sec:bounds} as they all have $avd(G)=\frac{2n}{3}$, where $n$ is the order of $G$). %Furthermore this family of graphs are tight for the upper bound for isolate-free graphs in Conjecture \ref{conj:2n/3}. Techniques in this paper may be applied to other graph polynomials with all real roots. Particularly the matching polynomial \cite{1972Heilmann} and the independence polynomial of claw-free graphs \cite{2007Chudnovsky} have both been shown to have all real roots.

This paper is structured as follows. In Section~\ref{sec:avdnorm} we determine the extremal graphs for the average order of dominating sets of graphs of order $n$. In Section~\ref{sec:bounds} we develop bounds for the average order of domination sets for connected graphs, as well as for trees.  Section~\ref{sec:denserandom} introduces a normalized version of the parameter, describes the distribution of these parameters, and considers the values for Erd\"{o}s-Renyi random graphs. Finally we conclude with some remarks.

\section{Extremal Graphs}\label{sec:avdnorm}

For a graph on $n$ vertices, it is clear that $\avd(G) \leq n$ as every dominating set has cardinality at most $n$. This bound is achieved by $\overline{K_n}$, and this graph is the unique extremal graph, as any other graph of order $n$ has a dominating set of size smaller than $n$. On the other hand, what about the \emph{minimum} value of $\avd(G)$ over all graphs of order $n$? As you might expect, the complete graph $K_n$ is the unique extremal graph in this case, but the argument will be more subtle, and that is what we shall pursue now. 

We shall first need some technical results about the average cardinality of sets in collection of sets.
Let $X$ be a nonempty finite set and $\mathcal{P}(X)$ its powerset. For any nonempty subset $\mathcal{A} \subseteq \mathcal{P}(X)$ we define the average order of $\mathcal{A}$, denoted $\av(\mathcal{A})$ to be 
$$\av(\mathcal{A}) = \frac{1}{|\mathcal{A}|}\sum_{A \in \mathcal{A}}|A|.$$

\begin{lemma}
\label{lem:AvgLower}
For a nonempty finite set $X$, let $\mathcal{A}  \subset \mathcal{B} \subseteq \mathcal{P}(X)$. Then 
$$\av(\mathcal{B}) \leq \av(\mathcal{A})  \mbox{~~if and only if~~}  \av(\mathcal{B}-\mathcal{A}) \leq \av(\mathcal{A}).$$
\end{lemma}

\begin{proof}
Set $S(\mathcal{C}) = \sum_{C \in \mathcal{C}}|C|$ for any $\mathcal{C} \subseteq \mathcal{P}(X)$. Then

\begin{align*}
& &  \av(\mathcal{B}) \leq & \av(\mathcal{A})\\
\Leftrightarrow& & \frac{S(\mathcal{B})}{|\mathcal{B}|} \leq & \frac{S(\mathcal{A})}{|\mathcal{A}|}  \hspace{2mm}  \\
\Leftrightarrow& & \frac{S(\mathcal{A}) + S(\mathcal{B} - \mathcal{A})}{|\mathcal{A}| + |\mathcal{B} - \mathcal{A}|} \leq & \frac{S(\mathcal{A})}{|\mathcal{A}|}  \hspace{2mm} \\
\Leftrightarrow& & S(\mathcal{A})|\mathcal{A}| + S(\mathcal{B} - \mathcal{A})|\mathcal{A}|   \leq & S(\mathcal{A})|\mathcal{A}| + S(\mathcal{A})|\mathcal{B} - \mathcal{A}|  \hspace{2mm}\\
\Leftrightarrow& & S(\mathcal{B} - \mathcal{A})|\mathcal{A}|   \leq & S(\mathcal{A})|\mathcal{B} - \mathcal{A}|  \hspace{2mm} \\
\Leftrightarrow& & \frac{S(\mathcal{B} - \mathcal{A})}{|\mathcal{B} - \mathcal{A}| }  \leq & \frac{S(\mathcal{A})}{|\mathcal{A}| } \hspace{2mm}  \\
\Leftrightarrow& & \av(\mathcal{B}-\mathcal{A}) \leq & \av(\mathcal{A}) \hspace{2mm}\\
\end{align*}
\end{proof}

\begin{lemma}
\label{lem:AvgPart}
For a nonempty finite set $X$, let $\mathcal{A} \subseteq \mathcal{P}(X)$. If there exists $r_1, r_2 \in \mathbb{R}$ and partition $\mathcal{A}_1, \mathcal{A}_2, \ldots, \mathcal{A}_k$ of $\mathcal{A}$ such that $r_1 \leq \av(\mathcal{A}_i)\leq r_2$, then $r_1 \leq \av(\mathcal{A})\leq r_2$.
\end{lemma}

\begin{proof}
Now

$$\av(\mathcal{A})=\frac{S(\mathcal{A})}{|\mathcal{A}|}=\frac{\sum_{i=1}^{k}S(\mathcal{A}_i)}{|\mathcal{A}|} = \frac{\sum_{i=1}^{k}|\mathcal{A}_i|\av(\mathcal{A}_i)}{|\mathcal{A}|} \geq  \frac{\sum_{i=1}^{k}|\mathcal{A}_i|r_1}{|\mathcal{A}|} = \frac{r_1\sum_{i=1}^{k}|\mathcal{A}_i|}{|\mathcal{A}|} =r_1$$
and 
$$\av(\mathcal{A})=\frac{S(\mathcal{A})}{|\mathcal{A}|}=\frac{\sum_{i=1}^{k}S(\mathcal{A}_i)}{|\mathcal{A}|} = \frac{\sum_{i=1}^{k}|\mathcal{A}_i|\av(\mathcal{A}_i)}{|\mathcal{A}|} \leq  \frac{\sum_{i=1}^{k}|\mathcal{A}_i|r_2}{|\mathcal{A}|} = \frac{r_2\sum_{i=1}^{k}|\mathcal{A}_i|}{|\mathcal{A}|} =r_2.$$

\end{proof}

 A \emph{simplicial complex} $\mathcal{A}$ is a subset of $\mathcal{P}(X)$ such that $\emptyset \in \mathcal{A}$ and $A \in \mathcal{A}$ implies $\mathcal{P}(A) \subseteq \mathcal{A}$. Simplicial complexes have numerous applications in combinatorics (and algebraic topology); here we will need a result on the average size of a set in a complex.

\begin{proposition}
\label{thm:simpialhalf}
Let $\mathcal{A}$ be a simplicial complex on a nonempty finite set $X$ with $n$ elements. Then for all $k \leq \frac{n}{2}$

$$|\mathcal{A}_{k}| \geq |\mathcal{A}_{n-k}|,$$

\noindent where $\mathcal{A}_{k}=\{A \in \mathcal{A} : |A|=k\}$. Hence $\av(\mathcal{A}) \leq \frac{n}{2}$.
\end{proposition}

\begin{proof}
We will use Hall's Theorem (see, for example, \cite{brualdi}) to show $|\mathcal{A}_{k}| \geq |\mathcal{A}_{n-k}|$. Consider the bipartite graph with bipartition $(\mathcal{A}_{n-k}$, $\mathcal{A}_{k})$ where $A\in \mathcal{A}_{n-k}$ and $B\in \mathcal{A}_{k}$ are adjacent if and only if  $B \subseteq A$. As $\mathcal{A}$ is a simplicial complex, the degree of each $A\in \mathcal{A}_{n-k}$ is $ {n-k \choose k}$ and the degree each $B\in \mathcal{A}_{k}$ is at most $ {n-k \choose n-2k}={n - k \choose k}$. Furthermore, for any subset $S \subseteq \mathcal{A}_{n-k}$ there are exactly $|S| {n-k \choose k}$ edges incident with the vertices of $S$ and at most $|N(S)|{n - k \choose k}$ edges incident to the vertices of $N(S)$. Therefore $|S| \leq |N(S)|$ and by Hall's Theorem you can match $\mathcal{A}_{n-k}$ into $\mathcal{A}_{k}$, so $|\mathcal{A}_{k}| \geq |\mathcal{A}_{n-k}|$.

Now let $\mathcal{B}_k  = \mathcal{A}_{n-k} \cup \mathcal{A}_{k}$. Note that $\av(\mathcal{B}_k) \leq \frac{n}{2}$ and $\mathcal{B}_1, \mathcal{B}_2, \ldots, \mathcal{B}_{\frac{n}{2}}$ is a partition of $\mathcal{A}$. It follows from Lemma \ref{lem:AvgPart} that $\av(\mathcal{A}) \leq \frac{n}{2}$.
\end{proof}

On our path to proving that $K_n$ is the unique graph, among all graphs of order $n$, with the least average order of dominating sets, we shall need the following that states that every subset of vertices that does not omit the closed neighbourhood of \emph{some} vertex must be dominating.
 
\begin{lemma}
\label{lem:MinDeg}
\textnormal{\cite{2010Char}} Let $G$ be a graph of order $n$ then  $d_{n-k}(G) = {n \choose k}$ for all $k \leq \delta(G)$, where $\delta(G)$ is the minimum degree of $G$.
\QED
\end{lemma}

\begin{theorem}
\label{cor:avdlowerbound}
Let $G$ be a graph of order $n$ then $\avd(G)\geq \frac{n2^{n-1} }{2^{n}-1 }$ with equality if and only if $G \cong K_n$.
\end{theorem}

\begin{proof}
Let $\overline{\mathcal{D}(G)}$ be the collection of subsets $S \subseteq V(G)$ such that $V(G)-S$ is a dominating set of $G$. Note $\overline{\mathcal{D}(G)}$ is a simplicial complex. Therefore by Theorem \ref{thm:simpialhalf} for all $k \leq \frac{n}{2}$,
$$d_{n-k} = |\{S \in \overline{\mathcal{D}(G)} : |S|=k\}| \geq |\{S \in \overline{\mathcal{D}(G)} : |S|=n-k\}| = d_k.$$
It follows that $\avd(G)$  is minimized if $d_{n-k} = d_k$ for all $k \leq \frac{n}{2}$. As $d_n=1$ and $d_0=0$ this cannot happen for $k = 0$, but 
$\avd(G)$ will be minimized if and only if $d_{n-k} = d_k$ for all $1 \leq k \leq \frac{n}{2}$, and this does occur for $K_n$. Thus $\avd(G) \geq \frac{n2^{n-1} }{2^{n}-1 } > \frac{1}{2}$. To find all extremal graphs,  suppose $d_{n-k} = d_k$ for all $1 \leq k \leq \frac{n}{2}$. 

First we assume that $\delta(G)\geq 1$. By Lemma \ref{lem:MinDeg}, $d_{n-1}=n$ and therefore $d_1=n$. A dominating set of order one must be a vertex of degree $n-1$. Therefore $G$ has $n$ vertices of degree $n-1$ and is hence $K_n$. Now suppose $\delta(G) = 0$ and $G$ has $r\geq1$ isolated vertices. Then $D(G,x) = x^rD(H,x)$ for some isolate-free graph on $n-r$ vertices. Furthermore

$$\avd(G) = \frac{D'(G,1)}{D(G,1)} = \frac{rD(H,1)+D'(H,1)}{D(H,1)} = r+\avd(H) \geq r+\frac{n-r}{2} \geq \frac{n+1}{2}.$$ 
If $n=1$, then $G \cong K_1$ and $\frac{n+1}{2} = \frac{n2^{n-1}}{2^{n}-1}$. An easy induction shows that $\frac{n+1}{2} > \frac{n2^{n-1}}{2^{n}-1}$ for $n \geq 2$, completing the proof.
%Thus it suffices to show $\frac{n+1}{2} > \frac{n2^{n-1}}{2^{n}-1}$ for $n \geq 2$. We will show this through induction. Note that
%
%$$\frac{n+1}{2} > \frac{n2^{n-1}}{2^{n}-1} \Leftrightarrow n2^{n}-n+2^n-1 > n2^{n} \Leftrightarrow 2^{n} > n+1.$$
%
%For $n=2$,  $2^2 = 4 > 3 = n+1$. Now suppose for some $k\geq2$, $2^{k} > k+1$. Then $2^{k+1}= 2 \cdot 2^{k} > 2(k+1) >k+2$. Therefore by induction $\frac{n+1}{2} = \frac{n2^{n-1}}{2^{n}-1}$ for all $n>2$.
\end{proof}

\section{Bounds}\label{sec:bounds}

\subsection{General graphs}

For a graph on $n$ vertices, we have seen that $\avd(G) \leq n$, with the bound achieved uniquely by $\overline{K_n}$. However, can we say more if we insist on the graph being connected? Or even just having no isolated vertices? The lower bound occurs for complete graphs, so no improvement is possible there, but the upper bound leaves some room for improvement. We shall do so first in terms of $\delta$, the minimum degree.

For a dominating set $S$ of a graph $G$ let 
$$a(S) = \{v \in S : S-v \notin \mathcal{D}(G)\},$$
the set of \emph{critical} vertices of $S$ with respect to domination (in that their removal makes the set no longer dominating). This parameter is key to improving the upper bound. we will need first an expression for the sum of $a(S)$ over all dominating sets.
 
\begin{lemma}
\label{lem:sumaS}
For a graph $G$ with $n$ vertices.
  
$$\sum\limits_{S \in \mathcal{D}(G)}|a(S)| = 2D'(G,1)-nD(G,1).$$
\end{lemma}

\begin{proof}
For a vertex $v \in V(G)$ let $a_v(G) = \{S \in\mathcal{D}(G) : S-v \notin \mathcal{D}(G)\}$, and let $\mathcal{D}_{+v}(G)$ denote the collection of dominating sets which contain $v$. Let $\mathcal{D}_{-v}(G)$ denote the collection of dominating sets which does not contain $v$. Clearly $a_v(G) \subseteq \mathcal{D}_{+v}(G)$. We will now show is a one-to-one correspondence between $\mathcal{D}_{+v}(G)-a_v(G)$ and $\mathcal{D}_{-v}(G)$. For any $S \in \mathcal{D}_{+v}(G)-a_v(G)$, $S-v \in \mathcal{D}(G)$ so clearly $S-v \in \mathcal{D}_{ -v}(G)$. Furthermore if $S\in \mathcal{D}_{-v}(G)$ then $S \cup \{v\} \in \mathcal{D}_{+v}(G)$ and $S \cup \{v\} \notin a_v(G)$. As the maps are injective, it follows that  $|\mathcal{D}_{+v}(G)-a_v(G)|=|\mathcal{D}_{-v}(G)|$ and as $a_v(G) \subseteq \mathcal{D}_{+v}(G)$ we have $|a_v(G)| = |\mathcal{D}_{+v}(G)|-|\mathcal{D}_{-v}(G)|$. Furthermore

\[\sum\limits_{v \in V(G)}|\mathcal{D}_{+v}(G)| = \sum_{i=1}^n i \cdot d(G,i) = D'(G,1)  \refstepcounter{eqcount} \label{eqn:dvi} \tag{\theeqcount}\]

\noindent and
 
\[\sum\limits_{v \in V(G)}|\mathcal{D}_{ -v}(G)| = \sum_{i=1}^n (n-i) \cdot d(G,i) = nD(G,1) - D'(G,1)  \refstepcounter{eqcount} \label{eqn:dvn-i} \tag{\theeqcount}.\]

\noindent Therefore

$$\sum\limits_{S \in \mathcal{D}(G)}|a(S)| = \sum\limits_{v \in V(G)}a_v(G) = \sum\limits_{v \in V(G)} \left(|\mathcal{D}_{+v}(G)|-|\mathcal{D}_{-v}(G)| \right)= 2D'(G,1)-nD(G,1).$$

\end{proof}

%\begin{proposition}
%Let $G$ be a graph on $n$ vertices. Then $\avd(G) \leq \frac{n+\Gamma(G)}{2}$ where $\Gamma(G)$ is the order of the largest minimal dominating set of $G$.
%\end{proposition}
%
%\begin{proof}
%For any dominating set $S$, there exists a minimal dominating set $M_S$ of $G$ such that $M_S \subseteq S$. Note that for any $v \in S-M_S$, $M_S \subseteq S-v$ and therefore $S-v$ is a dominating set. Hence $a(S) \subseteq M_S$ and so $|a(S)| \leq |M_S| \leq \Gamma(G)$. By Lemma \ref{lem:sumaS}
%
%$$2D'(G,1)-nD(G,1) = \sum\limits_{S \in \mathcal{D}(G)}|a(S)| \leq \sum\limits_{S \in \mathcal{D}(G)}\Gamma(G) = \Gamma(G)D(G,1).$$
%
%\noindent Therefore $\avd(G) = \frac{D'(G,1)}{D(G,1)} \leq \frac{n+\Gamma(G)}{2}$ and we obtain our result.
%\end{proof}
%
%JIB -- GIVE AN EXAMPLE _ IS THIS TIGHT?
%

In order to get to our upper bound, we need to partition $a(S)$. Let $S$ be a dominating set of $G$ containing the vertex $v$. By definition $v \in a(S)$ if and only if $S-v$ is not a dominating in $G$. Therefore $v \in a(S)$ if and only if there exists $u \in N[v]$ such that among the vertices of $S$, $u$ is only dominated by $v$ ($u$ could very well be $v$). We will call such a vertex $u$ a \emph{private neighbour} of $v$ with respect to $S$. Let $\Priv{S}{v}$ denote the collection of all private neighbours of $v$ with respect to $S$, that is,
$$\Priv{S}{v} = \{ u \in N[v] : N[u] \cap S = \{v\}\}.$$
Note $v \in a(S)$ if and only if $\Priv{S}{v} \neq \emptyset$. Moreover, for $v \in a(S)$, note that $\Priv{S}{v} \cap S \subseteq \{v\}$. We now partition $a(S) = a_1(S) \cup a_2(S)$, where 

\vspace{-6mm}

\begin{align*}
a_1(S) &= \{v \in a(S) : \Priv{S}{v} \cap (V-S) \neq \emptyset \}\\
a_2(S) &= \{v \in a(S) : \Priv{S}{v} =\{v\} \}.
\end{align*}
(We allow either to be empty.) Note that if $v \in a_2(S)$ then $N(v) \subseteq V-S$. We can partition $V-S = N_1(S) \cup N_2(S)$, where 

\vspace{-6mm}

\begin{align*}
N_1(S) &= \{v \in V-S : |N[v] \cap S| = 1 \}\\
N_2(S) &= \{v \in V-S : |N[v] \cap S| \geq 2 \}.
\end{align*}
That is, $N_1(v)$ is the set of those vertices outside of $S$ that have a single neighbour in $S$, and $N_2(S)$ are those that have more than one neighbour in $S$. (Again, we allow either to be empty.) 

As an example consider the labelled $P_5$ in Figure \ref{fig:P5}. Now $S=\{v_2,v_3,v_5\}$ is a dominating set. Furthermore $a(S)=\{v_2, v_5\}$ with $a_1(S)=\{v_2\}$, $a_2(S)=\{v_5\}$, $N_1(S) = \{v_1\}$ and $N_2(S) = \{v_4\}$. Alternatively, let $S'=\{v_1, v_3, v_5\}$. Now $a(S')=\{v_1, v_3, v_5\}$ with $a_1(S')=\emptyset$, $a_2(S')=\{v_1, v_3, v_5\}$, $N_1(S') = \emptyset$ and $N_2(S') = \{v_2,v_4\}$.

\begin{figure}[h]
\def\c{0.7}
\centering
\scalebox{\c}{
\begin{tikzpicture}
\begin{scope}[every node/.style={circle,thick,draw,fill}]
    \node[label=below:$v_1$](1) at (0,0) {}; 
    \node[label=below:$v_2$](2) at (1,0) {}; 
    \node[label=below:$v_3$](3) at (2,0) {}; 
    \node[label=below:$v_4$](4) at (3,0) {}; 
    \node[label=below:$v_5$](5) at (4,0) {}; 
\end{scope}

\begin{scope}
    \path [-] (1) edge node {} (2);
	\path [-] (2) edge node {} (3);
	\path [-] (3) edge node {} (4);
	\path [-] (4) edge node {} (5);
\end{scope}
\end{tikzpicture}}
\caption{A vertex labelled $P_5$}%
\label{fig:P5}%
\end{figure}
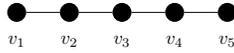

\begin{lemma}
\label{lem:a1n1}
Let $G$ be a graph. For any $S \in \mathcal{D}(G)$, $|a_1(S)| \leq |N_1(S)|.$
\end{lemma}

\begin{proof}
For any $u \in N_1(S)$, $N[u] \cap S \in a_1(S)$. Therefore the map $f:N_1(S) \rightarrow a_1(S)$ where $f(v) = N[v] \cap S$ is surjective, so $|N_1(S)| \geq |a_1(S)|$.
\end{proof}

For graph $G$ containing a vertex $v$ let $p_v(G)$ denote the collection of subsets of $V-N[v]$ which dominate $G-v$ (and hence they dominate $G-N[v]$ as well). We are now ready to improve our upper bound for graph with no isolated vertices. 

\begin{theorem}
\label{thm:3/4}
Let $G$ be a graph with $n \geq 2$ vertices and minimum degree $\delta \geq 1$. Then 
$$\avd(G) \leq \frac{2n(2^{\delta}-1)+n}{3(2^{\delta}-1)+1},$$
and so  $\avd(G) \leq \frac{3n}{4}$.
\end{theorem}

\begin{proof}

We begin by showing

\[\sum_{S \in \mathcal{D}(G)}|a_2(S)|= \sum_{v \in V(G)}|p_v(G)|  \refstepcounter{eqcount} \label{eqn:a2} \tag{\theeqcount}.\]

It suffices to show for any $S \in \mathcal{D}(G)$ containing $v$, $v \in a_2(G)$ if and only if $S-v \in p_v(G)$. Suppose $v \in a_2(S)$. By definition of $a_2(G)$, $v \in a_2(G)$ if and only if $\Priv{S}{v} = \{v\}$. Therefore $N[v] \cap S = \{v\}$ and $N(v) \subseteq N_2(S)$. Furthermore $S-v \subseteq V-N[v]$ and dominates $G-v$ and thus $S-v \in p_v(G)$. Conversely, suppose $S-v \in p_v(G)$. By definition of $p_v(G)$, $S-v \notin \mathcal{D}(G)$ but $S \in \mathcal{D}(G)$. Therefore $v \in a(S)$. However every neighbour of $v$ is already dominated by $S-v$; therefore, $N(v) \subseteq N_2(S)$ and $v \in a_2(G)$.

For now fix $v \in V(G)$. By definition every $S \in p_v(G)$ dominates $G-v$ but does not contain any vertices of $N[G]$. Therefore for any non-empty $T \subseteq N(v)$, $S \cup T \in \mathcal{D}_{ -v}(G)$. Furthermore for $S_1, S_2 \in p_v(G)$ and $T_1, T_2 \in N(v)$ if $S_1 \cup T_1 = S_2 \cup T_2$ then $S_1 = S_2$ and $T_1 = T_2$. Let $\mathcal{P}( N(v))$ denote the power set of $N(v)$. Then there is an injective map from $p_v(G) \times (\mathcal{P}( N(v))-\emptyset)$ to $\mathcal{D}_{ -v}(G)$ and hence $(2^{\deg(v)}-1)|p_v(G)| \leq |\mathcal{D}_{-v}(G)|$. So together with $(\ref{eqn:dvn-i})$ and $(\ref{eqn:a2})$ we obtain

$$\sum_{S \in \mathcal{D}(G)}|a_2(S)| = \sum_{v \in V(G)}|p_v(G)| \leq \sum_{v \in V(G)}\frac{|\mathcal{D}_{-v}(G)|}{2^{\deg(v)}-1} \leq \sum_{v \in V(G)}\frac{|\mathcal{D}_{-v}(G)|}{2^{\delta}-1} = \frac{nD(G,1)-D'(G,1)}{2^{\delta}-1}.$$

By Lemma \ref{lem:a1n1}, $|a_1(S)| \leq |N_1(S)|$. So together with $(\ref{eqn:dvn-i})$ we obtain
$$\sum_{S \in \mathcal{D}(G)}|a_1(S)| \leq \sum_{S \in \mathcal{D}(G)}|N_1(S)| \leq \sum_{S \in \mathcal{D}(G)}|V-S|=\sum_{v \in V(G)}|\mathcal{D}_{-v}(G)| = nD(G,1)-D'(G,1).$$
By Lemma \ref{lem:sumaS},  $\sum\limits_{S \in \mathcal{D}(G)}|a(S)| = 2D'(G,1)-nD(G,1),$ and hence from 
\[ \sum_{S \in \mathcal{D}(G)} |a(S)| = \sum_{S \in \mathcal{D}(G)} |a_1(S)| + \sum_{S \in \mathcal{D}(G)} |a_2(S)|\]
we have that
$$2D'(G,1)-nD(G,1) \leq nD(G,1)-D'(G,1) + \frac{nD(G,1)-D'(G,1)}{2^{\delta}-1}.$$
From this it follows that 
$$\frac{D'(G,1)}{D(G,1)} \leq \frac{2n(2^{\delta}-1)+n}{3(2^{\delta}-1)+1}.$$

Finally, one can verify that as $\delta \geq 1$, 
\[ \frac{2n(2^{\delta}-1)+n}{3(2^{\delta}-1)+1} \leq \frac{3n}{4},\]
and we are done.
\end{proof}

%***START HERE***

Theorem \ref{thm:3/4} shows all graphs with no isolated vertices have $\avd(G) \leq \frac{3n}{4}$. However for $\delta \geq 4$ the bound can be  improved again, if we are even more careful with our counting. Again, we shall need a couple of technical lemmas first.

\begin{lemma}
\label{lem:edgeidentity}
For any graph $G$,

$$\sum_{S \in \mathcal{D}(G)} |N_1(S)| = \sum_{e \in E(G)} |\mathcal{D}(G)-\mathcal{D}(G-e)|.$$
\end{lemma}

\begin{proof}

It suffices to show for every dominating set $S \in \mathcal{D}(G)$ there are exactly $N_1(S)$ edges $e = \{u,v\}$ in $G$ such that $S \notin \mathcal{D}(G-e)$. For every $S \in \mathcal{D}(G)$ consider the edge $e$ in $G$. If $e$ goes from a vertex $v \in N_1(S)$ to some vertex in $S$ then $v$ is not dominated by $S$ in $G-e$, so $S \in \mathcal{D}(G)-\mathcal{D}(G-e)$. 

Conversely suppose $e$ does not go from a vertex in $N_1(S)$ to some vertex in $S$; we need to show that $S \notin \mathcal{D}(G)-\mathcal{D}(G-e)$. Note that in $G-e$, $S$ necessarily dominates every vertex other than possibly $u$ and $v$. Therefore $S\in \mathcal{D}(G-e)$ if and only if $S$ dominates both $u$ and $v$ in $G-e$. Consider the following 3 cases:

\vspace{3mm}
\noindent \emph{Case 1:} $u,v \in S$. Then both $u$ and $v$ dominate themselves in $S$ so $S$ is a dominating set in $G-e$. Therefore $S \notin \mathcal{D}(G)-\mathcal{D}(G-e)$.
\vspace{1mm}

\vspace{1mm}
\noindent \emph{Case 2:} $u,v \notin S$. As $S$ is a dominating set of $G$, there exists vertices $x,y \in S$ (possibly $x=y$) such that $x$ and $y$ are adjacent to $u$ and $v$ respectively in $G$. Note $x$ and $y$ are still adjacent to $u$ and $v$ respectively in $G-e$. Therefore $S$ is a dominating set in $G-e$ and $S \notin \mathcal{D}(G)-\mathcal{D}(G-e)$.
\vspace{1mm}

\noindent \emph{Case 3:} Either $u \in S$ and $v \notin S$, or $u \notin S$ and $v \in S$. Without loss of generality suppose $u \in S$ and $v \notin S$. As $e$ does not go from a vertex in $N_1(S)$ to some vertex in $S$ then $v \notin N_1(S)$ and therefore $v \in N_2(S)$. By definition of $N_2(S)$, there exists at least one other vertex $x \in S$ adjacent to $v$. Therefore $x$ is still adjacent to $v$ in $G-e$ and $S$ is a dominating set in $G-e$. Therefore $S \notin \mathcal{D}(G)-\mathcal{D}(G-e)$.
\vspace{3mm}

Therefore for every dominating set $S \in \mathcal{D}(G)$, the number of edges $e$ in $G$ which have $S \in \mathcal{D}(G)-\mathcal{D}(G-e)$ is exactly the number of edges from $N_1(S)$ to $S$. By definition of $N_1(S)$, each vertex in $N_1(S)$ is adjacent to exactly one vertex in $S$. Therefore, the number of edges $e$ in $G$ which have $S \in \mathcal{D}(G)-\mathcal{D}(G-e)$ is exactly $|N_1(S)|$.

\end{proof}

\begin{lemma}
\label{lem:edgeremoval}
\textnormal{\cite{2012Recurr}} Let $G$ be a graph. For every edge $e=\{u,v\}$ of $G$, 

$$|\mathcal{D}(G)-\mathcal{D}(G-e)|=|p_u(G-e)|+|p_v(G-e)|-|p_u(G)|-|p_v(G)|.$$
\QED
\end{lemma}

We are now ready to prove another upper bound for $\avd(G)$.

\begin{theorem}
\label{thm:upperbound}
For any graph $G$ with no isolated vertices,

$$\avd(G) \leq \frac{n}{2}+\sum_{v \in V(G)} \frac{\deg(v)}{2^{\deg(v)+1}-2}.$$
\end{theorem}

\begin{proof}
By Lemma \ref{lem:a1n1}, Lemma \ref{lem:edgeidentity}, Lemma \ref{lem:edgeremoval} we obtain.

\vspace{-4mm}

$$\sum_{S \in \mathcal{D}(G)}|a_1(S)| \leq \sum_{e \in E(G)}(|p_u(G-e)|+|p_v(G-e)|-|p_u(G)|-|p_v(G)|) = \sum_{v \in V(G)}\sum_{u \in N(v)}(|p_v(G-uv)|-|p_v(G)|)$$

\noindent Together with $(\ref{eqn:a2})$ we obtain

$$\sum_{S \in \mathcal{D}(G)}|a(S)| =\sum_{S \in \mathcal{D}(G)}(|a_1(S)|+|a_2(S)|) \leq \sum_{v \in V(G)}\sum_{u \in N(v)}|p_v(G-uv)|-\sum_{v \in V(G)}(\deg(v)-1)|p_v(G)|.$$

\noindent Furthermore as $G$ has no isolated vertices we obtain

\[\sum_{S \in \mathcal{D}(G)}|a(S)| \leq \sum_{v \in V(G)}\sum_{u \in N(v)}|p_v(G-uv)| \refstepcounter{eqcount} \label{eqn:asum} \tag{\theeqcount}.\]

\noindent For each $v \in V(G)$ and $e = \{u,v\} \in E(G)$ consider $S \in p_v(G-e)$. For any nonempty $T \subseteq N_G[v]-\{u\}$, $S \cup T \in \mathcal{D}(G-e) \subseteq \mathcal{D}(G)$ and all such sets are distinct. Therefore $(2^{\deg(v)}-1)|p_v(G-e)| \leq |\mathcal{D}(G)|$ (where the degree is in the graph $G$)  and together with Lemma \ref{lem:sumaS} and $(\ref{eqn:asum})$ we obtain
$$2D'(G,1)-nD(G,1)=\sum_{S \in \mathcal{D}(G)}|a(S)| \leq \sum_{v \in V(G)}\frac{\deg(v) \cdot D(G,1)}{2^{\deg(v)}-1},$$
from which is follows that 
$$ \frac{D'(G,1)}{D(G,1)} \leq \frac{n}{2}+\sum_{v \in V(G)} \frac{\deg(v)}{2^{\deg(v)+1}-2}.$$
\end{proof}

\begin{corollary}
\label{cor:upperbound}
For a graph $G$ with minimum degree $\delta \geq 1$.

$$\avd(G) \leq \frac{n}{2} \left(1+\frac{\delta}{2^{\delta}-1} \right).$$

\noindent In particular if $\delta \geq 2 \log_2(n)$ then $\avd(G) \leq \frac{n+1}{2}.$
\end{corollary}

\begin{proof}
Let $f(x)=\frac{x}{2^{x+1}-2}$. It is not hard to verify that for $x \geq 1$, $f(x)$ is a decreasing function. Therefore for all $v \in V(G)$, $f(\deg(v))\leq f(\delta)$, and by Theorem \ref{thm:upperbound}
$$\avd(G) \leq \frac{n}{2}+\sum_{v \in V(G)} \frac{\deg(v)}{2^{\deg(v)+1}-2} \leq \frac{n}{2}+\frac{n \cdot \delta}{2^{\delta+1}-2} =\frac{n}{2} \left(1+\frac{\delta}{2^{\delta}-1} \right).$$
Now suppose $\delta \geq 2 \log_2(n)$. As $\delta \leq n-1$, we know that $2\log_2(n) \leq n-1$. Again, one can verify that $g(n) = \delta/(2^\delta-1)$ is decreasing for $\delta \geq 1$, so
\begin{eqnarray*}
\avd(G) & \leq & \frac{n}{2} \left(1+\frac{\delta}{2^{\delta}-1} \right) \\
 & \leq & \frac{n}{2} \left(1+\frac{2\log_2(n)}{2^{2\log_2(n)}-1} \right) \\
  & \leq &  \frac{n}{2} \left(1+\frac{n-1}{n^2-1}\right) \\ 
   & = & \frac{n}{2} \left(1+\frac{1}{n+1}\right)\\
    & \leq & \frac{n}{2} \left(1+\frac{1}{n}\right) \\
    & = & \frac{n+1}{2}.
\end{eqnarray*}
\end{proof}

Theorem \ref{thm:3/4} and Corollary \ref{cor:upperbound} give two different upper bounds for $\avd(G)$ based on $\delta(G)$. Figure \ref{fig:MDO89} plots $\avd(G)$ sorted by minimum degree for all graphs of order $n=8$ and $n=9$, respectively. The curve in Figure \ref{fig:MDO89} is the minimum of the two bounds of Theorem \ref{thm:3/4} and Corollary \ref{cor:upperbound} evaluated for each integer $0 \leq \delta \leq n$ and linearly interpolated between each point.

\begin{figure}[!h]
\centering
\subfigure[Graphs of order 8]{\includegraphics[scale=0.35]{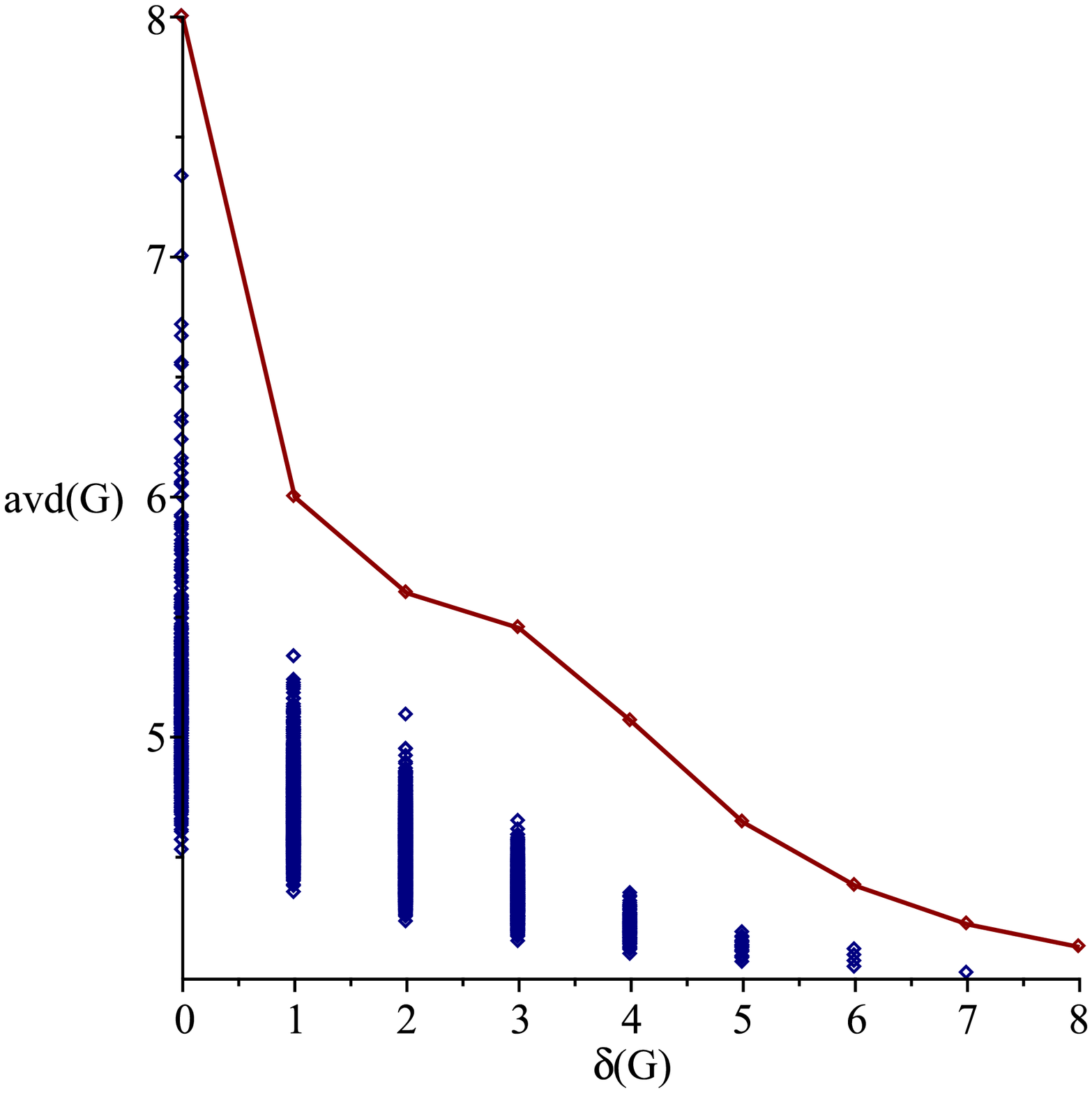}}
\qquad
\subfigure[Graphs of order 9]{\includegraphics[scale=0.35]{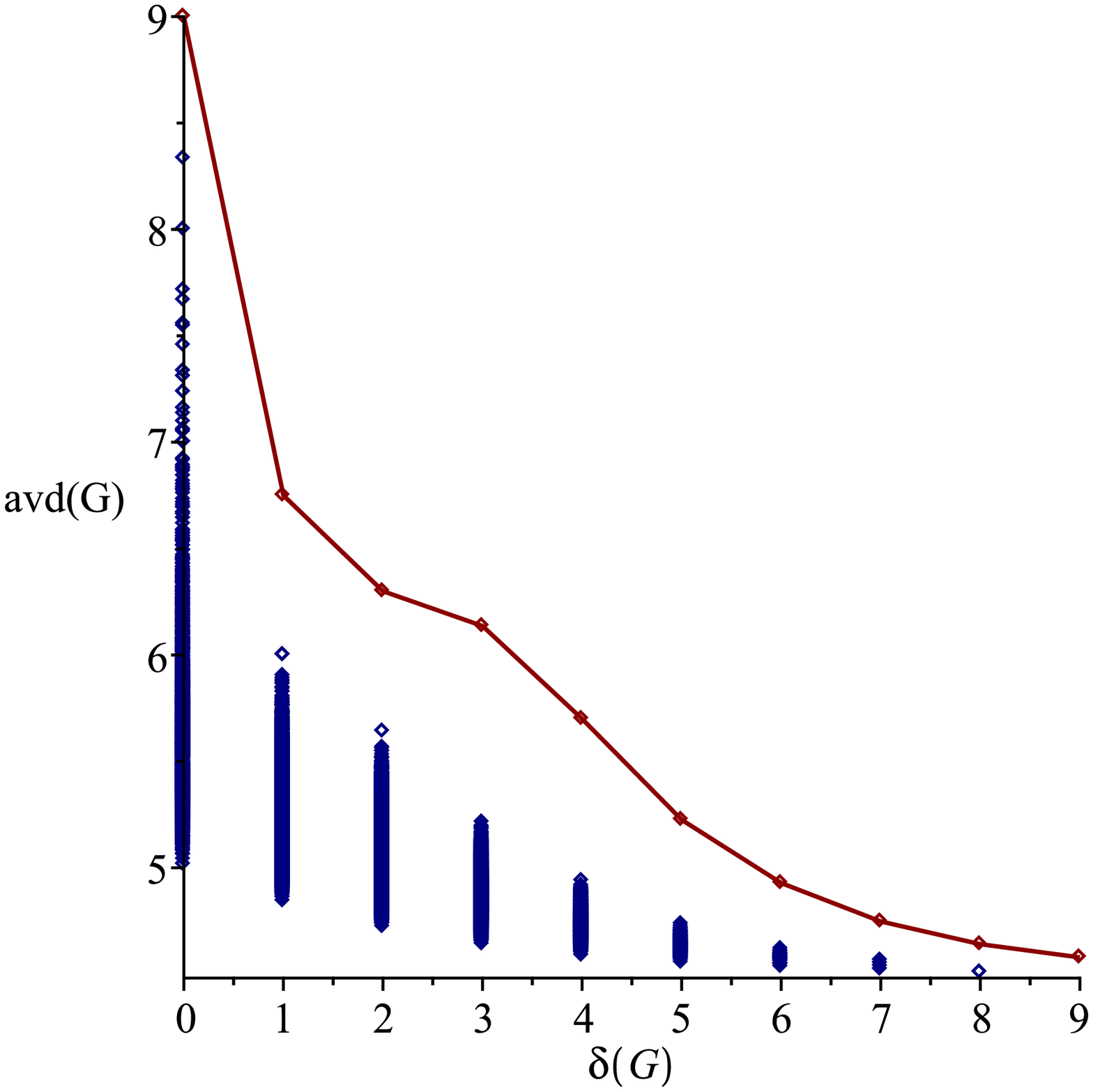}}
\caption{The bounds from Theorem \ref{thm:3/4} and Corollary \ref{cor:upperbound} compared to $\avd(G)$ for $n=8$ and $n=9$.}%
\label{fig:MDO89}%
\end{figure}

Our best upper bound for all isolate-free graphs remains $\avd(G) \leq \frac{3n}{4}$. However by Corollary \ref{cor:upperbound} if $\delta(G) \geq 4$ then $\avd(G) \leq \frac{19n}{30} \leq \frac{2n}{3}$. In fact, all graphs up to order 9 with no isolated vertices have $\avd(G) \leq \frac{2n}{3}$. This leads us to the following conjecture.

\vspace{0.2in}

\begin{conjecture}
\label{conj:2n/3}
Let $G$ be a graph with $n \geq 2$ vertices. If $G$ has no isolated vertices (so, in particular, if $G$ is connected) then $\avd(G) \leq \frac{2n}{3}$.
\end{conjecture}

\vspace{0.2in}

We can show that the upper bound in Conjecture \ref{conj:2n/3} is achieved for all $n\geq 2$: For $n=2$ and $n=3$, $\avd(K_2)=\frac{4}{3}$ and $\avd(K_{1,2})=2$. For any $n \geq 4$ there exists non-negative integers $k$ and $\ell$ such that $n=2k+3\ell$. Then by Lemma \ref{lem:disjointunion} any graph of the form $H=kK_2 \cup \ell K_{1,2}$ will have $\avd(H)=\frac{2n}{3}$. These graphs are not connected, but one can insist on connectivity as follows. Let $G$ be any graph on $k+ \ell$ vertices, and let $G'$ be the graph obtained by adding one leaf to $k$ vertices of $G$ and two leaves to the other $\ell$ vertices of $G$. Oboudi showed \cite{2015Oboudi} that $D(G',x)=D(H,x)$. Therefore $\avd(G')=\frac{2n}{3}$, and by choosing $G$ to be connected, the graph $G'$ will be as well. 
%In \cite{2015Oboudi} Oboudi conjectured this family of graphs were the only graphs with domination polynomials that have all real roots.

While we are unable to prove Conjecture~\ref{conj:2n/3}, we can provide some evidence for it. A graph $G$ is called \emph{quasi-regularizable} if one can replace each edge of $G$ with a non-negative number of parallel copies, so as to obtain a regular multigraph of minimum degree at least one. Any graph which contains a spanning subgraph which is both regular and nonempty is quasi-regularizable; in particular, any graph which contain either a perfect matching or a hamiltonian cycle is quasi-regularizable. Berge \cite{1980berge} characterized quasi-regularizable graphs as those for which $|S| \leq |N(S)|$ holds for any independent set $S$ of $G$. We will now show that for quasi-regularizable graphs,\ Conjecture~\ref{conj:2n/3} holds. 

\begin{theorem}
\label{thm:boundquasi}
If $G$ is a quasi-regularizable graph then $\avd(G) \leq \frac{2n}{3}$.
\end{theorem}

\begin{proof}
We begin by showing $|a(S)| \leq n-|S|$ for every $S \in \mathcal{D}(G)$. By Lemma \ref{lem:a1n1}, $|a_1(S)| \leq |N_1(S)|$. Therefore it suffices to show   $|a_2(S)| \leq |N_2(S)|$. For every $v \in a_2(S)$, $N(v) \subseteq V-S$ as otherwise $\Priv{S}{v} \neq \{v\}$. Furthermore, $N(v) \subseteq N_2(S)$ as otherwise $v \in a_1(S)$. Therefore $a_2(S)$ is an independent set with $N(a_2(S))\subseteq N_2(S)$. As $G$ is a quasi-regularizable graph then $|a_2(S)| \leq |N(a_2(S))| \leq |N_2(S)|$, so
\[ |a(S)| = |a_1(S)| + |a_2(S)| \leq |N_1(S)| + |N_2(S)| = n-|S|.\] 

Finally, as $|a(S)| \leq n-|S|$ then $\sum\limits_{S \in \mathcal{D}(G)}|a(S)| \leq nD(G,1)-D'(G,1)$. Thus together with Lemma \ref{lem:sumaS} we obtain
$$2D'(G,1)-nD(G,1) \leq nD(G,1)-D'(G,1) \Rightarrow \avd(G) = \frac{D'(G,1)}{D(G,1)} \leq \frac{2n}{3}.$$
\end{proof}

\subsection{Trees}
\label{sec:trees}
In this section we turn to trees (which are connected and, if are nontrivial, have $\delta \geq 1$). For every $n\geq 2$ there is a tree $T$ of order $n$ with $\avd(T) = \frac{2n}{3}$, satisfying the upper bound from Conjecture \ref{conj:2n/3} for isolate-free graphs. Examples of trees which achieve the upper bound given in Conjecture \ref{conj:2n/3} are described in the paragraph following Conjecture \ref{conj:2n/3} (if one chooses the base graph $G$ there to be a tree as well). It also remains an open question whether this is actual the upper bound amongst trees. 

%We conclude this section by constructing families of graphs $(G_n)_{n \geq 1}$, namely paths and cycles, with $\lim_{n \rightarrow \infty} \navd(G_n) =r \approx 0.618419922$.

However, what about the lower bound? For graphs the lower bound was achieved by complete graphs, but these are far from being trees. We show now that  $\avd(T) \geq \avd(K_{1,n-1})$, and the argument is even more involved than for the lower bound for general graphs. For this we require a result similar to that of Proposition~\ref{thm:simpialhalf}. However the proof of this is considerably more involved. For a tree $T$ of order $n$, recall $\mathcal{D}(T)$ denotes the collection of all dominating sets in $T$. For now fix $S \in \mathcal{D}(T)$. Recall in the proof of Proposition~\ref{thm:simpialhalf}, it was important to bound the number of subsets $S' \subseteq S$ where $S'$ is also a dominating set and $|S'|=k$. Let 

$$\mbox{dom}_k(S)=|\{S' \subseteq S: S' \in \mathcal{D}(T) \text{ and }|S'|=k\}|.$$

The trivial upper bound, which was used in the proof of Proposition~\ref{thm:simpialhalf}, is simply $\mbox{dom}_k(S) \leq {|S| \choose k}$, but we need something stronger for trees. Recall $a(S)=\{v \in S:S-v \notin \mathcal{D}(T)\}$. Therefore for any $S' \subseteq S$, if $S' \in \mathcal{D}(T)$ then $a(S) \subseteq S'$. Therefore $\mbox{dom}_k(S) \leq {|S|-|a(S)| \choose k-|a(S)|}$. However this is only useful if $a(S) \neq \emptyset$. On the other hand, when $a(S)=\emptyset$, $S$ is \emph{double dominating set} \cite{2000Harary}, that is, a subset $S \subseteq V(G)$ such that for every vertex $v \in V(G)$, $|N[v] \cap S| \geq 2$. The order of the smallest double dominating set is denoted $\gamma_{\times 2}(G)$. Note that for a dominating set $S$ of a tree $T$, if $|S| < \gamma_{\times 2}(T)$ then $a(S) \geq 1$. Moreover, suppose $\gamma(T)+\gamma_{\times 2}(T) \geq n+1$. Then $|S| > \gamma_{\times 2}(T)$ implies $n+1-|S|<\gamma(T)$ and hence $\mbox{dom}_k(S) =0$ for $k \leq n+1-|S|$.

%Double domination has been a subject to a fair amount of research since 2000 (see \cite{2005Harant,2005Henning,2006Blidia,2006DoubleDom2Dom,2008Chang,2019Hajian} ).

\begin{theorem}
\label{thm:doubledombound}
If $T$ is a nontrivial tree then $\gamma_{\times 2}(T)+\gamma(T) \geq n+1$.
\end{theorem}

\begin{proof}
We can assume that $n \geq 3$, as if $n = 2$, then $T = K_2$ and so $\gamma_{\times 2}(T) = 2$, $\gamma(T)=1$ and the result holds.
Set $V(T)=V$. It is sufficient to show for any double dominating set $S$, $\gamma(T) \geq n-|S|+1$. Note that if the $m$ vertices $v_1, \ldots, v_m$ had pairwise disjoint closed neighbourhoods, then $\gamma(T) \geq m$ as any dominating set would need to contain at least one vertex from each closed neighbourhood. Therefore it is sufficient to show for any double dominating set $S$, there exists a collection of $|V-S|+1$ vertices with pairwise disjoint closed neighbourhoods. We will induct on the number of vertices in $V-S$. For $v \in V$ and $u \in N(v)$ let $B(T,v,u)$ denote the set of vertices in the same component as $u$ in $T-v$ (See Figure \ref{fig:Branch}).
 
 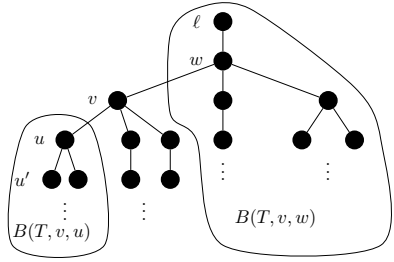
\begin{figure}[h]
\def\c{0.7}
\centering
\scalebox{\c}{
\begin{tikzpicture}

 \draw[] plot[smooth cycle] coordinates
 {(-3.75,-1.75) (-2.5,-1.75) (-2.25,-4) (-4,-4)};
\node[] at (-3.25,-3.75) {$B(T,v,u)$};

 \draw[] plot[smooth cycle] coordinates
 {(-0.5,0.5) (0.5,0.5) (2.5,-1) (3,-2) (3,-4) (0,-4) (-0.5,-2) (-1,-1.5) (-1,0)};
 \node[] at (1,-3.5) {$B(T,v,w)$};
 
\begin{scope}[every node/.style={circle,thick,draw,fill}]
    \node[label=left:$\ell$](1) at (0,0.25) {}; 
    \node[label=left:$w$](2) at (0,-0.5) {}; 
    \node[label=left:$v$](3) at (-2,-1.25) {}; 
    \node[](4) at (0,-1.25) {}; 
    \node[](5) at (2,-1.25) {};
    
    \node[label=left:$u$](6) at (-3,-2) {}; 
    \node[](7) at (-1.75,-2) {}; 
    \node[](8) at (-1,-2) {}; 
    \node[](9) at (0,-2) {}; 
    \node[](10) at (1.5,-2) {};
    \node[](11) at (2.5,-2) {};
    
    \node[label=left:$u'$](12) at (-3.25,-2.75) {}; 
    \node[](13) at (-2.75,-2.75) {}; 
    
    \node[](14) at (-1.75,-2.75) {}; 
    \node[](15) at (-1,-2.75) {}; 
\end{scope}

\begin{scope}
    \path [-] (1) edge node {} (2);
	\path [-] (2) edge node {} (3);
	\path [-] (2) edge node {} (4);
	\path [-] (2) edge node {} (5);
	\path [-] (3) edge node {} (6);
	\path [-] (3) edge node {} (7);
	\path [-] (3) edge node {} (8);
	\path [-] (4) edge node {} (9);
	\path [-] (5) edge node {} (10);
	\path [-] (5) edge node {} (11);
	\path [-] (6) edge node {} (12);
	\path [-] (6) edge node {} (13);
	\path [-] (7) edge node {} (14);
	\path [-] (8) edge node {} (15);
\end{scope}

\node at ($(-3,-3)!.25!(-3,-4)$) {\vdots};
\node at ($(-1.5,-3)!.25!(-1.5,-4)$) {\vdots};
\node at ($(0,-2)!.5!(0,-3)$) {\vdots};
\node at ($(2,-2)!.5!(2,-3)$) {\vdots};

\end{tikzpicture}}
\caption{An example of $B(T,v,u)$ and $B(T,v,w)$.}%
\label{fig:Branch}%
\end{figure}

 Let $S$ be a double dominating set, so that $S$ contains every leaf and stem of $T$ (as \emph{stem} in a tree is a vertex adjacent ot a leaf). The case where $|V-S|=0$ is vacuously true for any nontrivial tree. Assume for any nontrivial tree and some $k \geq 0$ if $|V-S|\leq k$ then there exists a collection of $|V-S|+1$ vertices with pairwise disjoint closed neighbourhoods. Now suppose $|V-S|=k+1$. Note that for any leaf in $T$, both it and its \emph{stem} (i.e. the leaf's only neighbour) must both be in $S$, otherwise $S$ is not a double dominating set. Fix a leaf $\ell \in V$. Now choose a vertex $v \notin S$ which is of maximum distance to $\ell$. Note $v$ is not a stem, as otherwise $v \in S$ which contradicts $v \notin S$. Furthermore, $\deg(v) \geq 2$, as otherwise either $v \in S$ or $S$ is not a double dominating set; in particular, $v \neq \ell$. 
 
 Let $w \in N(v)$ be the only neighbour of $v$ which is closer $\ell$ than $v$ (See Figure \ref{fig:Branch}). Note that $w \neq \ell$, as otherwise $v$ would be stem and hence belong to $S$. As $\deg(v) \geq 2$, choose $u \in N(v)-\{w\}$. Note every vertex in $B(T,v,u)$ is further from $\ell$ than $v$ and therefore $B(T,v,u)\subseteq S$. Moreover, as $v$ is not a stem, $\deg(u)\geq 2$. Therefore choose $u' \in N(u)-\{v\}$. Note that $N[u'] \subseteq B(T,v,u)$ (as in  Figure~\ref{fig:Branch}). Now set $T' = B(T,v,w)$ and $S'=S \cap B(T,v,w)$. $T'$ is a nontrivial tree as $w,\ell \in T'$. For each $x \in V(T')$, $N_T[x] = N_{T'}[x]$, except for $w$ where $N_T[x] = N_{T'}[x] \cup \{v\}$. As $v \notin S$, $|N_{T}[x] \cap S| =|N_{T'}[x] \cap S'| \geq 2$ for all $x \in V(T')$. Therefore $S'$ is a double dominating set of $T'$. Finally the only vertex in $V(T)-V(T')$ which was not in $S$ was $v$, as $v$ was the furthest vertex from $\ell$ which was not in $S$. Therefore $|V(T')-S'| = |V(T)-S|-1=k$ and by our induction hypothesis there exists a collection of $k+1$ vertices with disjoint closed neighbourhoods in $T'$. Let $P$ denote this collection. As $v \notin N_T[u']$ then $N_{T}[x] \cap N_{T}[u']=\emptyset$ for all $x \in V(T')$. Therefore $P \cup \{u'\}$ is a collection of $k+2 = |V-S|+1$ vertices with pairwise disjoint closed neighbourhoods in $T$.
\end{proof}

We need two additional lemmas on the way to finding the extremal tree with the least average order of dominating sets.

\begin{lemma}
\label{lem:MinDomSetsTrees}
\textnormal{\cite{2006DomNumTree}} $K_{1,n-1}$ has the most dominating sets amongst all trees of order $n$.
\QED
\end{lemma}

\begin{lemma}
\label{lem:DoubleDom2Dom}
\textnormal{\cite{2006DoubleDom2Dom}} For every nontrivial tree $T$, $2\gamma(T) \leq \gamma_{\times 2}(T)$.
\QED
\end{lemma}

\begin{lemma}
\label{lem:treehalfplus1}
If $T$ is a graph with $n$ vertices then $d_{n-k} \geq d_{k+1}$ for all $k+1 \leq \frac{n+1}{2}$.
\end{lemma}

\begin{proof}
%We will first show $d(T,n-k) \geq d(T,k+1)$ for all $k+1 \leq \frac{n+1}{2}$. 

Fix $k  \leq \frac{n+1}{2}$. If $k+1 < \gamma(T)$ then clearly $d_{n-k} \geq d_{k+1}$ holds as $d_{k+1}=0$. So suppose for the remiander of this proof that  $k+1 \geq \gamma(T)$. We will now use Hall's Theorem again. As before, let $\mathcal{D}_k$ denote the collection of all dominating sets of order $k$. We now construct a bipartite graph with bipartition $(\mathcal{D}_{k+1}$, $\mathcal{D}_{n-k})$; two vertices $A\in \mathcal{D}_{k+1}$ and $B\in \mathcal{D}_{n-k}$ are adjacent if $A \subseteq B$. As every superset of a dominating set remains dominating, the degree of each $A\in \mathcal{D}_{k+1}$ is ${n-k-1 \choose n-2k-1}={n-k-1 \choose k}$. By the same argument used in the proof of Proposition~\ref{thm:simpialhalf}, it suffices to show for any $B \in \mathcal{D}_{n-k}$ there are at most ${n-k-1 \choose k}$ subsets of $B$ which are in $\mathcal{D}_{k+1}$.

By Theorem \ref{thm:doubledombound}, $\gamma_{\times 2}(T)+\gamma(T) \geq n+1$ and hence $k+1 \geq n+1-\gamma_{\times 2}(T)$ We now consider two cases:

\vspace{2mm}

\noindent \emph{Case 1:} Suppose $k+1 > n+1-\gamma_{\times 2}(T)$ then $\gamma_{\times 2}(T) > n-k$. For any dominating set $B\in \mathcal{D}_{n-k}$ there exists a vertex $v \in T$ such that $N[v] \cap B$ contains exactly one vertex. Let $\{u\} = N[v] \cap B$. Then $B-u$ is no longer a dominating set. Hence there are at most ${n-k \choose k+1} - {n-k-1 \choose k+1} = {n-k-1 \choose k}$ subsets of $B$ which are also in $\mathcal{D}_{k+1}$.

\vspace{2mm}

\noindent \emph{Case 2:} Suppose $k+1= n+1-\gamma_{\times 2}(T)$ then $\gamma_{\times 2}(T) = n-k$. As $\gamma_{\times 2}(T)+\gamma(T) \geq n+1$ then $\gamma(T) \geq k+1$. Furthermore, as $\gamma(T) \leq k+1$, it follows that $k+1=\gamma(T)$. For any dominating set $B\in  \mathcal{D}_{n-k}$, if $B$ is not a double dominating set then it follows for Case 1 that there are at most ${n-k-1 \choose k}$ subsets of $B$ which are also in $\mathcal{D}_{k+1}$. So suppose $B$ is a double dominating set. Let $m$ be the number stems in $T$. If $m=1$ then $T=K_{1,n-1}$. It is easy to see $k+1 = \gamma(K_{1,n-1})=1$ so $n-k=n$. Furthermore more $d_n(K_{1,n-1}) = d_1(K_{1,n-1})=1$ and therefore $d_n(K_{1,n-1}) \geq d_1(K_{1,n-1})$. Now suppose $m \geq 2$. Choose two stems $s_1$ and $s_2$ along with leaves $\ell_1$ and $\ell_2$ which are adjacent to $s_1$ and $s_2$ respectively. As $B$ is a double dominating set then $s_1, s_2, \ell_1, \ell_2 \in B$, otherwise $\ell_1$ or $\ell_2$ with not be double dominated. Furthermore if $A \subseteq B$ such that $A \in \mathcal{D}_{k+1}$, then $A$ is a minimum dominating set. Therefore $A$ contains exactly one of $s_i$ or $\ell_i$ for each $i=1,2$ and remaining $k+1$ vertices of $A$ are chosen from then remaining $n-k-4$ vertices in $B$. Therefore there are at most $4{n-k-4 \choose k-1}$ subsets $A \subseteq B$ such that $A \in \mathcal{D}_{k+1}$. If $n-k-4 \geq 1$, then ${n-k-4 \choose k-1}={n-k-5 \choose k-2}+{n-k-5 \choose k-1} \leq {n-k-4 \choose k-2} +{n-k-4 \choose k}$ and

\begin{align*}
4{n-k-4 \choose k-1}  \leq & \left({n-k-4 \choose k-2} +{n-k-4 \choose k-1} \right)+{n-k-4 \choose k-1}+\left({n-k-4 \choose k-1}+{n-k-4 \choose k} \right) \\ 
=& {n-k-3 \choose k-2}+{n-k-4 \choose k-1}+{n-k-3 \choose k} \\
\leq&\left( {n-k-3 \choose k-2}+{n-k-3 \choose k-1}\right)+\left({n-k-3 \choose k-1}+{n-k-3 \choose k}\right) \\
=& {n-k-2 \choose k-1}+{n-k-2 \choose k}  \\
=& {n-k-1 \choose k}, \\
\end{align*}
 
\noindent Otherwise, suppose $n-k-4 < 1$. As $\gamma_{\times 2}(T) = n-k$ then $\gamma_{\times 2}(T) \leq 4$. By Lemma \ref{lem:DoubleDom2Dom}, $2\gamma(T) \leq \gamma_{\times 2}(T)$. Therefore $\gamma(T) \leq 2$. Furthermore as $T$ has two stems, $\gamma(T) \geq 2$ and therefore $\gamma(T) = 2$. Now $\gamma_{\times 2}(T)+\gamma(T) = n-k + k + 1 = n+1$, so $n \leq 5$. There are exactly three trees with $\gamma(T) = 2$ and $n\leq 5$. They are shown below.

\begin{figure}[!h]
\def\c{0.55}
\centering
\subfigure[$P_4$]{
\scalebox{\c}{
 \begin{tikzpicture}
\begin{scope}[every node/.style={circle,thick,draw}]
    \node (1) at (0,0) {};
    \node (2) at (1,0) {};
    \node (3) at (2,0) {};
    \node (4) at (3,0) {};
\end{scope}

\begin{scope}
    \path [-] (1) edge node {} (2);
    \path [-] (2) edge node {} (3);
    \path [-] (3) edge node {} (4);
\end{scope}    
\end{tikzpicture}}}
\qquad
\subfigure[$T_5$]{
\scalebox{\c}{
 \begin{tikzpicture}
\begin{scope}[every node/.style={circle,thick,draw}]
    \node (1) at (0,0) {};
    \node (2) at (1,0) {};
    \node (3) at (2,0) {};
    \node (4) at (3,0.5) {};
    \node (5) at (3,-0.5) {};
\end{scope}

\begin{scope}
    \path [-] (1) edge node {} (2);
    \path [-] (2) edge node {} (3);
    \path [-] (3) edge node {} (4);
    \path [-] (3) edge node {} (5);
\end{scope}
\end{tikzpicture}}}
\qquad
\subfigure[$P_5$]{
\scalebox{\c}{
 \begin{tikzpicture}
\begin{scope}[every node/.style={circle,thick,draw}]
    \node (1) at (0,0) {};
    \node (2) at (1,0) {};
    \node (3) at (2,0) {};
    \node (4) at (3,0) {};
    \node (5) at (4,0) {};
\end{scope}

\begin{scope}
    \path [-] (1) edge node {} (2);
    \path [-] (2) edge node {} (3);
    \path [-] (3) edge node {} (4);
    \path [-] (4) edge node {} (5);
\end{scope}
\end{tikzpicture}}}
%\caption{Examples of complete graphs}%
%\label{fig:EgCommonFamKn}%
\end{figure}

\noindent However $\gamma_{\times 2}(P_4)=4$, $\gamma(P_4) = 2$ so $\gamma_{\times 2}(P_4)+\gamma(P_4) = 6 \neq n+1$ and $\gamma_{\times 2}(T_5)=5$, $\gamma(T_5) = 2$ so $\gamma_{\times 2}(T_5)+\gamma(T_5) = 7 \neq n+1$, so we can omit these cases. Finally, $D(P_5,x) = x^5+5x^4+8x^3+3x^2$ which satisfies $d_{n-k} \geq d_{k+1}$ for $k+1 \leq \frac{n+1}{2}$.

%It now remains to show $\avd(T) \geq \frac{n+1}{2}$. For $k+1 \leq \frac{n+1}{2}$ let $\mathcal{B}_k  = \mathcal{D}_{n-k}(T) \cup \mathcal{D}_{k+1}(T)$. Note that $\av(\mathcal{B}_k) \geq \frac{n+1}{2}$ as $d(T,n-k) \geq d(T,k+1)$. Furthermore  $\mathcal{B}_1, \mathcal{B}_2, \ldots, \mathcal{B}_{\frac{n+1}{2}}$ is a partition of $\mathcal{D}(T)$. It follows from Lemma \ref{lem:AvgPart} that $avd(T) = \av(\mathcal{D}(T))  \geq \frac{n+1}{2}$.
\end{proof}

\begin{theorem}
If $T$ is a graph with $n$ vertices $\avd(T)\geq \avd(K_{1,n-1})$ with equality if and only if $T \cong K_{1,n-1}$.
\end{theorem}

\begin{proof}
By Lemma \ref{lem:treehalfplus1}, $\avd(T) \geq \frac{n+1}{2}$ and $d_{n-k} \geq d_{k+1}$ for $k+1 \leq \frac{n+1}{2}$. Suppose $T \not\cong K_{1,n-1}$, so $\gamma(T) \geq 2$ and $d(T,1)=0$. Now consider the mean domination order of all dominating sets except for the dominating set $V(T)$. Let $\mathcal{D}^{\ast}(T) = \mathcal{D}(T)-\{V(T)\}$. Note that

$$\av(\mathcal{D}^{\ast}(T)) = \frac{D'(T,1)-n}{D(T,1)-1}$$

\noindent For $k+1 \leq \frac{n+1}{2}$ let $\mathcal{B}_k  = \mathcal{D}_{n-k}(T) \cup \mathcal{D}_{k+1}(T)$. Note that $\av(\mathcal{B}_k) \geq \frac{n+1}{2}$ as $d_{n-k} \geq d_{k+1}$. Furthermore  $\mathcal{B}_2, \ldots, \mathcal{B}_{\frac{n+1}{2}}$ is a partition of $\mathcal{D}^{\ast}(T)$. It follows from Lemma \ref{lem:AvgPart} that 

$$\frac{D'(T,1)-n}{D(T,1)-1} = \av(\mathcal{D}^{\ast}(T))  \geq \frac{n+1}{2}.$$

\noindent By Lemma \ref{lem:MinDomSetsTrees}, $D(T,1) \leq D(K_{1,n-1},1)$ and

\begin{align*}
\avd(T) &=  \frac{n+D'(T,1)-n}{D(T,1)} \\
       &=  \frac{n}{D(T,1)}+\left(\frac{D(T,1)-1}{D(T,1)}\right)\frac{D'(T,1)-n}{D(T,1)-1}  \\
       &\geq  \frac{n}{D(T,1)}+\left(\frac{D(T,1)-1}{D(T,1)}\right)\frac{n+1}{2}  \\
       &\geq  \frac{n}{D(K_{1,n-1},1)}+\left(\frac{D(K_{1,n-1},1)-1}{D(K_{1,n-1},1)}\right)\frac{n+1}{2}  \\
       &>  \frac{n-1}{D(K_{1,n-1},1)}+\left(\frac{D(K_{1,n-1},1)-1}{D(K_{1,n-1},1)}\right)\frac{n+1}{2}  \\
       &= \frac{n-1}{2^{n-1}+1}+\left(\frac{2^{n-1}}{2^{n-1}+1}\right)\frac{n+1}{2}  \\
       &= \frac{n-1+2^{n-2}(n+1)}{2^{n-1}+1}=   \avd(K_{1,n-1})
\end{align*}

\end{proof}

\section{Distribution of Average Order of Dominating Sets}\label{sec:denserandom}

What are the possible values for $\avd(G)$? If $G$ is a graph of order $n$, we showed in the previous section $\avd(G) \in (\frac{n}{2},n]$, but it seems unlikely that one can say precisely what values in the interval are average orders of dominating sets. A natural variant of $\avd(G)$ is $\navd(G) = \frac{\avd(G)}{n}$ which we shall refer to as the \emph{normalized average order of dominating sets in $G$}. (Similar kinds of normalized graph parameters have been investigated throughout the literature -- for example, \cite{1983Jamison,2010Vince, 2014Haslegrave}.) 

We start with some examples. We say a graph contains a \emph{simple $k$-path} if there exists $k$ vertices of degree two which induce a path in $G$. Two families of graphs which contains simple $k$-paths are paths $P_n$ and cycles $C_n$ (where $k = n-2$ and $n-1$, respectively). The following holds for graphs which contain simple 3-paths.

\begin{theorem}
\label{thm:DomPolyIP}
\textnormal{\cite{2012Recurr}} Suppose $G$ is a graph with vertices $u, v, w$ which form a simple 3-path. Then 
$$D(G,x)=x(D(G/u,x) + D(G/u/v,x) + D(G/u/v/w,x))$$
\noindent where $G/u$ is the graph formed by joining every neighbour of $u$ and then deleting $u$.
\QED
\end{theorem}

There is no known closed formula for all coefficients of $D(P_k,x)$ and $D(C_k,x)$ respectively. This makes determining mean dominating order of paths and cycles difficult. We will now show that for a family of graphs satisfying a recurrence relation similar to that in Theorem~\ref{thm:DomPolyIP}, we can calculate the limit of the normalized average order of dominating sets as $n \rightarrow \infty$. First we shall put forward a way to calculate the limits of average values of functions of a certain type (which include those that arise  from solving linear polynomial recurrences); the proof is straightforward and omitted.

\begin{theorem}
\label{thm:MDORecurr}
Suppose functions $f_n(x)$ satisfy 

$$f_n(x) = \alpha_1(x) (\lambda_1(x))^n + \alpha_2(x) (\lambda_2(x))^n + \cdots + \alpha_k(x) (\lambda_k(x))^n$$

where $\alpha_i(x)$ and  $\lambda_i(x)$ are fixed non-zero analytic functions, such that $|\lambda_1(1)|>|\lambda_i(1)|$ for all $i > 1$. Then

$$\lim_{n \rightarrow \infty} \frac{f'_n(1)}{nf_n(1)}=\frac{\lambda_1'(1)}{\lambda_1(1)}.$$
\QED
\end{theorem}

%\begin{proof}
%As $|\lambda_1(1)|>|\lambda_i(1)|$ for all $i > 1$ then $\lim_{n \rightarrow \infty} \frac{\lambda_i(1)^n}{\lambda_1(1)^n}=0$ for all $i > 1$. Furthermore
%
%\begin{align*}
%\lim_{n \rightarrow \infty} \frac{f'_n(1)}{nf_n(1)} & = \lim_{n \rightarrow \infty} \frac{\sum\limits_{i=1}^k \left( \alpha_i'(1) \lambda_i(1)^n+n\alpha_i(1) \lambda_i(1)^{n-1}\lambda_i'(1) \right)}{n\sum\limits_{i=1}^k\alpha_i(1) \lambda_i(1)^n} \\
%& = \lim_{n \rightarrow \infty} \frac{\sum\limits_{i=1}^k \frac{ \alpha_i'(1) \lambda_i(1)^n+n\alpha_i(1) \lambda_i(1)^{n-1}\lambda_i'(1)}{\lambda_1(1)^n}}{n\sum\limits_{i=1}^k  \frac{\alpha_i(1) \lambda_i(1)^n}{\lambda_1(1)^n}} \\
%%& = \lim_{n \rightarrow \infty} \frac{\sum\limits_{i=1}^k \frac{ \alpha_i'(1) \lambda_i(1)^n+n\alpha_i(1) \lambda_i(1)^{n-1}\lambda_i'(1)}{\lambda_1(1)^n}}{n\sum\limits_{i=1}^k  \frac{\alpha_i(1) \lambda_i(1)^n}{\lambda_1(1)^n}} \\
%& = \lim_{n \rightarrow \infty} \frac{\alpha_1'(1)+\frac{n\alpha_1(1)\lambda_1'(1)}{\lambda_1(1)}}{n\alpha_1(1)} \\
%& = \lim_{n \rightarrow \infty} \frac{\alpha_1'(1)}{n\alpha_1(1)}+\frac{\lambda_1'(1)}{\lambda_1(1)} = \frac{\lambda_1'(1)}{\lambda_1(1)}
%\end{align*}
%\end{proof}

\begin{theorem}
$\lim\limits_{n \rightarrow \infty} \navd(P_n)=\lim\limits_{n \rightarrow \infty} \navd(C_n) \approx 0.618419922.$
\end{theorem}

\begin{proof}
For both paths and cycles, we have a sequence of graphs $(G_n)_{n \geq 1}$ which satisfy 
$$D(G_n,x)=x(D(G_{n-1},x) + D(G_{n-2},x) + D(G_{n-3},x))$$
As $G_n$ follows the homogeneous linear recursive relation $D(G_n,x)=x(D(G_{n-1},x) + D(G_{n-2},x) + D(G_{n-3},x))$, then $D(G_n,x) = \alpha_1(x) \lambda_1(x)^n + \alpha_2(x) \lambda_2(x)^n + \alpha_3(x) \lambda_3(x)^n$ where each $\lambda_1(x)$ satisfies

$$\lambda_i(x)^3-x\lambda_i(x)^2-x\lambda_i(x)-x=0.$$

\noindent We solve this cubic polynomial (see also \cite{P4free}). The solutions are
\begin{align*}
\lambda_1(x) =& \frac{x}{3}+p(x)+q(x), \\
\lambda_2(x) =& \frac{x}{3}-p(x)-q(x) + \frac{\sqrt{3}}{2}\left( p(x)-q(x) \right)i, \\
\lambda_3(x) =& \frac{x}{3}-p(x)-q(x) - \frac{\sqrt{3}}{2}\left( p(x)-q(x) \right)i,
\end{align*}
where
\begin{align*}
p(x) =& \sqrt[3]{\frac{x^3}{27}+\frac{x^2}{6}+\frac{x}{2}+\sqrt{\frac{x^4}{36}+\frac{7x^3}{54}+\frac{x^2}{4}}} \\
q(x) =& \sqrt[3]{\frac{x^3}{27}+\frac{x^2}{6}+\frac{x}{2}-\sqrt{\frac{x^4}{36}+\frac{7x^3}{54}+\frac{x^2}{4}}}
\end{align*}

Note $|\lambda_1(1)| \approx 1.83929 > |\lambda_2(1)|=|\lambda_3(1)| \approx 0.73735$. 

Therefore by Theorem \ref{thm:MDORecurr}, $\lim\limits_{n \rightarrow \infty} \frac{\avd(G_n)}{n}=\frac{\lambda_1'(1)}{\lambda_1(1)}$. 
It follows that  
\vspace{-7mm} 
 
\begin{align*}
\lim\limits_{n \rightarrow \infty} \frac{\avd(G_n)}{n} &=\frac{\lambda_1'(1)}{\lambda_1(1)} \\
&= \frac{\frac{1}{3}+p'(1)+q'(1)}{\frac{1}{3}+p(1)+q(1)} \\
&= \frac{\frac{1}{3}+\frac{27\sqrt{33}+187}{66(19+3\sqrt{33})^\frac{2}{3}}-\frac{27\sqrt{33}-187}{66(19-3\sqrt{33})^\frac{2}{3}}}{\frac{1}{3}+\frac{(19+3\sqrt{33})^\frac{1}{3}}{3}+\frac{(19-3\sqrt{33})^\frac{1}{3}}{3}} \\
&=\frac{1}{3}+\frac{(88-8\sqrt{33})(19+3\sqrt{33})^{\frac{1}{3}}}{1056}+\frac{(55-7\sqrt{33})(19+3\sqrt{33})^{\frac{2}{3}}}{1056},
\end{align*}
which we will denote by $r$.
 By Theorem \ref{thm:DomPolyIP}, both $C_n$ and $P_n$ satisfy the same recurrence as $G_n$ and hence $\lim\limits_{n \rightarrow \infty} \frac{\avd(P_n)}{n}=\lim\limits_{n \rightarrow \infty} \frac{\avd(C_n)}{n}=r \approx 0.618419922$.
\end{proof}

For all graphs of order 9 we counted the number of graphs with $\navd(G) \in [\frac{1}{2}+\frac{k}{20n},\frac{1}{2}+\frac{k+1}{20n})$ for each integer $0 \leq k \leq 10n-1$. Figure~\ref{fig:MDO9dist} shows the linearly interpolated distribution of $\navd(G)$ for all graphs of order 9. The distribution appears to be skewed towards $\frac{1}{2}$. However, our next result shows $\navd(G)$ can be arbitrarily closed to any value in $\left[ \frac{1}{2},1 \right]$. 

\begin{center}
\begin{figure}[!h]
\centering
\includegraphics[scale=0.35]{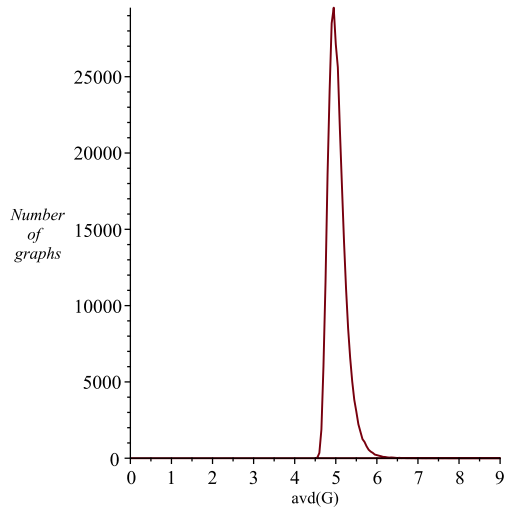}
\caption{Distribution of $\avd(G)$ for all graphs of order 9}%
\label{fig:MDO9dist}%
\end{figure}
\end{center}

\begin{proposition}
The set $\left\{\navd(G) : G \text{ is a graph}\right\}$ is dense in $[\frac{1}{2},1]$.
\end{proposition}

\begin{proof}
It suffices to show for every rational number $\frac{a}{b} \in [0.5,1]$ ($a$ and $b$ positive) there exists a sequence of graphs $(G_k)_{k \geq 1}$, with orders $n_k$ respectively, such that $\lim_{k \rightarrow \infty} n_{k} = \infty$ and $\lim_{k \rightarrow \infty}\frac{\avd(G_k)}{n_k} = \frac{a}{b}$. Let $G_k = (2b-2a)K_k \cup (2a-b)\overline{K_k}$; $G_k$ has order $(2b-2a)k + (2a-b)k = bk$. Note such a graph exists as $\frac{a}{b} \in [0.5,1]$ and hence $a \leq b \leq 2a$. Recall $D(K_k,x)=(1+x)^k-1$ and $D(\overline{K_k},x)=x^k$. Therefore 

$$\lim_{i \rightarrow \infty} \frac{\avd(K_k)}{k}=\lim_{k \rightarrow \infty} \frac{k2^{k-1}}{k(2^k-1)}= 0.5 \hspace{2mm}\text{ and }\hspace{2mm}\lim_{k \rightarrow \infty} \frac{\avd(\overline{K_k})}{k}=1. $$

\noindent Therefore by Lemma \ref{lem:disjointunion}

\begin{align*}
\lim_{k \rightarrow \infty} \frac{\avd(G_k)}{bk} & = \lim_{k \rightarrow \infty} \frac{(2b-2a)\avd(K_k) + (2a-b)\avd(\overline{K_k})}{bk} \\
&= \lim_{i \rightarrow \infty} \frac{(2b-2a)\avd(K_k)}{bk} +\lim_{k \rightarrow \infty} \frac{(2a-b)\avd(\overline{K_k})}{bk}\\
&= \frac{(2b-2a) \cdot 0.5}{b} +\frac{2a-b}{b} = \frac{a}{b}
\end{align*}
\end{proof}

While we have shown that the closure of the normalized average order of dominating sets fills the interval $[1/2,1]$, where do most values lie? Let ${\mathcal G}(n,p)$ denote the sample space of random graphs on $n$ vertices (each edge exists is independent present with probability $p$). We will now show with probability tending to $1$, the normallized average order of dominating sets of a random graph approaches $\frac{1}{2}$ (even if the graph is sparse with $p$ close to $0$); this explains the ``bundling up'' of values near $n/2$ in Figure~\ref{fig:MDO9dist}. 

\begin{theorem}
\label{thm:almostallgotohalf}
Let $G_{n} \in {\mathcal G}(n,p)$. Then 
\[ \lim\limits_{n\rightarrow \infty} \navd(G_n)=\frac{1}{2}.\]
\end{theorem}

\begin{proof}
It follows from Theorem~\ref{cor:avdlowerbound} that $\navd(G_n) \geq \frac{1}{2}$. Therefore it is sufficient to show $\lim\limits_{n\rightarrow \infty} \navd(G_n) \leq \frac{1}{2}$. 

The degree of any vertex $v$ of $G_n$ has a binomial distribution $X_v$ with $N = n-1$, and hence has mean $p(n-1)$. From Hoeffding's well known bound on the tail of a binomial distribution, it follows that for any fixed $\varepsilon > 0$,
\[ \mbox{Prob}\left( X_v \leq (p-\varepsilon)(n-1) \right) \leq e^{-2\varepsilon^2(n-1)}.\]
Thus 
\[ \mbox{Prob}\left( \cup_{v} X_v \leq (p-\varepsilon)(n-1) \right) \leq ne^{-2\varepsilon^2(n-1)} \rightarrow 0.\]
It follows that $\delta(G_n) > (p-\varepsilon)(n-1) > 2 \log_2(n)$ with probability tending to $1$.
By Corollary \ref{cor:upperbound}, if $\delta(G_n) \geq 2 \log_2(n)$ then $\avd(G_n) \leq \frac{n+1}{2}$. Therefore with probability tending to $1$,
$$\lim\limits_{n\rightarrow \infty} \navd(G_n) \leq \lim\limits_{n\rightarrow \infty} \frac{n+1}{2n}= \frac{1}{2},$$
and we are done.
\end{proof}

\section{Conclusion and Open Problems}\label{sec:conclusion}

The most salient open problem is that in Conjecture~\ref{conj:2n/3}, namely, a tight upper bound on the average order of dominating sets among all connected graphs of order $n$. As mentioned previously, any graph which contains a perfect matching is quasi-regularizable. Let $\nu(G)$ denote the matching number of $G$, that is, the largest cardinality of a matching. We alter the proof of the previous theorem to put $\avd(G)$ in terms of $\nu(G)$. This will not improve the bound from Theorem~\ref{thm:boundquasi} for graphs with perfect matchings. However there are graphs which contain near perfect matchings which are not quasi-regularizable and therefore not subject to the bound in Theorem \ref{thm:boundquasi}, for example paths of odd order. However, we can get an upper bound via the matching number.

\begin{theorem}
Let $G$ be a graph with $n$ vertices. Then $\avd(G) \leq n-\frac{2\nu(G)}{3}$.
\end{theorem}

\begin{proof}
We begin by showing $|a(S)| \leq 2(n-\nu(G))-|S|$ for every $S \in \mathcal{D}(G)$. By Lemma \ref{lem:a1n1}, $|a_1(S)| \leq |N_1(S)|$. Therefore it suffices to show $|a_2(S)| \leq |N_2(S)|+n-2\nu(G)$. For every $v \in a_2(S)$, $N(v) \subseteq V-S$ otherwise $\Priv{S}{v} \neq \{v\}$. Furthermore $N(v) \subseteq N_2(S)$ otherwise $v \in a_1(S)$. Fix a maximum matching in $G$. Each vertex in $a_2(S)$ is either unmatched or matched with a vertex in $N_2(S)$. Note there are at most $n-2\nu(G)$ unmatched vertices in $G$. Therefore $|a_2(S)| \leq |N_2(S)|+n-2\nu(G)$.

Finally as $|a(S)| \leq 2(n-\nu(G))-|S|$ then $\sum\limits_{S \in \mathcal{D}(G)}|a(S)| \leq 2(n-\nu(G))D(G,1)-D'(G,1)$. Thus together with Lemma \ref{lem:sumaS} we obtain

$$2D'(G,1)-nD(G,1) \leq 2(n-\nu(G))D(G,1)-D'(G,1) \Rightarrow \avd(G) = \frac{D'(G,1)}{D(G,1)} \leq n-\frac{2\nu(G)}{3}.$$
\end{proof}

Another avenue of research is investigating the monotonicity of $\avd(G)$ with respect to vertex or edge deletion. For example the removal of any edge or vertex in a graph decreases the number of dominating sets. However this is not necessarily the case for $\avd(G)$. Let $G$ be the graph pictured in Figure \ref{fig:Graphev}.

\begin{figure}[h]
\def\c{0.7}
\centering
\scalebox{\c}{
\begin{tikzpicture}
\begin{scope}[every node/.style={circle,thick,draw,fill}]
    \node[label=above:$v_1$](1) at (1,1) {}; 
    \node[label=below:$v_2$](2) at (1,-1) {}; 
    \node[label=right:$v_3$](3) at (1,0) {}; 
    \node[label=below:$v_4$](4) at (0,0) {}; 
    \node[label=above:$v_5$](5) at (-1,1) {}; 
    \node[label=below:$v_6$](6) at (-1,-1) {}; 
\end{scope}

\begin{scope}
    \path [-] (1) edge node {} (4);
	\path [-] (2) edge node {} (4);
	\path [-] (3) edge node {} (4);
	\path [-] (5) edge node {} (4);
	\path [-] (6) edge node {} (4);
	\path [-] (5) edge node {} (6);
\end{scope}
\end{tikzpicture}}
\caption{A vertex labelled graph}%
\label{fig:Graphev}%
\end{figure}
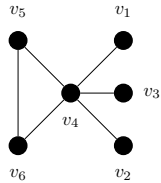

It is not difficult to determine $D(G,x) = x^6+6x^5+12x^4+10x^3+5x^2+x$ and therefore $\avd(G) = \tiltfrac{25}{7}$. However

\begin{itemize}
\item $\avd(G-v_1) = \frac{58}{19}<\avd(G) < \frac{13}{3} = \avd(G-v_1)$
\item $\avd(G-v_5v_6) = \frac{39}{11}<\avd(G) < \frac{78}{19} = \avd(G-v_1v_4)$
\end{itemize}

\noindent Despite this the following conjecture holds for all graphs on up to 7 vertices

\begin{conjecture}
For a nonempty graph $G$ there exists a vertex $v$ and edge $e$ such that

$$\avd(G-v)<\avd(G)<\avd(G-e).$$
\end{conjecture}

\bibliographystyle{plain}
%\bibliographystyle{abbrv}
%\bibliographystyle{ieeetr}
%\bibliographystyle{amsplain2}
%\bibliography{C:/Users/Iain/Desktop/Math/BibTeX/MeanDomOrder}%{}
\bibliography{MeanDomOrder}%{}

\end{document}